\documentclass[reqno]{amsart}

\usepackage{amssymb,amsmath,amsthm}
\usepackage[capitalise]{cleveref}
\usepackage{enumitem}
\usepackage{xparse}
\usepackage[cmtip,arrow,matrix,line]{xy}

\newtheorem{theorem}{Theorem}[section]
\newtheorem*{thmnonum}{Theorem}
\newtheorem{proposition}[theorem]{Proposition}
\newtheorem{lemma}[theorem]{Lemma}
\newtheorem{corollary}[theorem]{Corollary}
\newtheorem*{cornonum}{Corollary}

\DeclareMathOperator{\add}{add}

\DeclareMathOperator{\End}{End}

\DeclareMathOperator{\Hom}{Hom}

\DeclareMathOperator{\pd}{pd}

\renewcommand{\mod}{\operatorname{mod}} 
\renewcommand{\H}{\textup{\textsf{H}}}

\newcommand{\A}{\mathcal{A}}
\newcommand{\C}{\mathcal{C}}
\newcommand{\D}{\mathcal{D}}

\newcommand{\I}{\mathcal{I}}
\newcommand{\T}{\mathcal{T}}
\newcommand{\X}{\mathcal{X}}
\newcommand{\Y}{\mathcal{Y}}
\newcommand{\Z}{\mathcal{Z}}
 
\NewDocumentCommand\set{m+g}{\left\{#1\IfNoValueTF{#2}{}{\;\left|\;#2\right.}\right\}}

\newenvironment{defenum}{\begin{enumerate}[label=(\arabic{*})]}{\end{enumerate}}
\newenvironment{thmenum}{\begin{enumerate}[label=(\alph{*})]}{\end{enumerate}}

\begin{document}

\title{Projective Dimensions in Cluster-Tilted Categories}
\author{Alex Lasnier}
\address{D\'epartement de math\'ematiques, Universit\'e de Sherbrooke, 2500 boul. de l'Universit\'e, Sherbrooke (Qu\'ebec), J1K 2R1, Canada.}
\email{alex.lasnier@usherbrooke.ca}


\begin{abstract}
We study the projective dimensions of the restriction of functors $\Hom_\C(-,X)$ to a contravariantly finite rigid subcategory $\T$ of a triangulated category $\C$.
We show that the projective dimension of $\left.\Hom_\C(-,X)\right|_\T$ is at most one if and only if there are no non-zero morphisms between objects in $\T[1]$ factoring through $X$, when the object $X$ belongs to a suitable subcategory of $\C$.
As a consequence, we obtain a characterisation of the objects of infinite projective dimension in the category of finitely presented contravariant functors on a cluster-tilting subcategory of $\C$.
\end{abstract}

\maketitle

%
\section{Introduction}
%

Cluster categories were introduced in \cite{BMRRT} to provide a categorical framework for the study of acyclic cluster algebras.
Let $\C$ be a cluster category.
An object $T$ in $\C$ is called a tilting object if the following holds: $\Hom_\C(T,X[1])=0$ if and only if $X\in\add T$.
These tilting objects provide a model for the combinatorics of the corresponding cluster algebra.
Cluster-tilted algebras were defined in \cite{BMR} as endomorphism algebras of the form $\End_\C(T)$ where $T$ is a tilting object in $\C$.
 
In \cite{BBT}, Beaudet, Br\"{u}stle and Todorov characterise the modules of infinite projective dimension over a cluster-tilted algebra by investigating certain ideals of $\End_\C(T[1])$. Such an ideal, denoted $I_M$, is given by the endomorphisms of $T[1]$ factoring through $M$, where $M$ is an object in $\C$. They show that the $\End_\C(T)$ module $\Hom_\C(T,M)$ has infinite projective dimension if and only if $I_M$ is non-zero.

In this paper, we seek to generalise this characterisation in the following setting: 
Let $\C$ be a $\Hom$-finite Krull-Schmidt triangulated category and let $\T$ be a contravariantly finite subcategory of $\C$.
It is well-known that $\mod\T$ the category of finitely presented contravariant functors on $\T$ is abelian and has enough projectives.
Thus, we can meaningfully talk about the projective dimension of an object $F$ in $\mod\T$, which we denote $\pd F$.

For a subcategory $\D$ of $\C$, let $\I_{X}(\D)$ be the ideal of $\D$ formed by the morphisms between objects in $\D$ factoring through an object $X$ in $\C$.
Our aim is to study the projective dimensions of the functors $\left.\Hom_\C(-,X)\right|_\T$ in terms of the factorisation ideals $\I_X$.
Our main result is the following:
\begin{thmnonum}
Let $\T$ be a contravariantly finite subcategory of $\C$ such that $\Hom_\C(\T,\T[1])=0$.
Let $X$ be an object in $\C$ having no direct summands in $\T[1]$. If there is a triangle
\[ \xymatrix{ T_1 \ar[r] & T_0 \ar[r] & X \ar[r] & T_1[1] } \]
with $T_0,T_1\in\T$, then 
\[ \pd \left.\Hom_\C(-,X)\right|_\T \leq 1 \text{ if and only if } \I_X(\T[1])=0. \]
\end{thmnonum}

The subcategory $\T$ is called cluster-tilting if $\Hom_\C(\T,X[1])=0$ is equivalent to $X\in\T$.
These subcategories are a natural generalisation of tilting objects in a cluster category.
As such, their properties have been investigated in several papers, for instance \cite{IY},\cite{KR},\cite{KR2},\cite{KZ}.
When $\T$ is a cluster-tilting subcategory of $\C$, we say that $\mod\T$ is a cluster-tilted category.
It was shown in \cite{KZ} that cluster-tilted categories are Gorenstein of dimension at most one.
Consequently, all objects have projective dimension zero, one or infinity.
Using this remark we obtain the following characterisation of the objects in $\mod\T$ having infinite projective dimension:  
\begin{cornonum}
Let $\T$ be a cluster-tilting subcategory of $\C$.
For every $X\in\C$ having no direct summands in $\T[1]$,
\[ \pd \left.\Hom_\C(-,X)\right|_\T = \infty \text{ if and only if } \I_X(\T[1])\neq 0. \]
\end{cornonum} 

In the last section, we explore some partial generalisations of these results when the subcategory $\T$ satisfies $\Hom_\C(\T,\T[i])=0$ for every $0<i<n$, or when $\T$ is $n$-cluster-tilting.

The author wishes to thank Louis Beaudet and Thomas Br\"{u}stle for several interesting discussions.

%
\section{Preliminaries}
%

Let $\C$ be a $\Hom$-finite Krull-Schmidt triangulated category.
Let $\T$ be a full subcategory of $\C$ closed under isomorphisms, direct sums and direct summands.
The subcategory $\T$ is said to be \emph{$n$-rigid} if
\[ \Hom_\C(\T,\T[i])=0 \text{ for all } 0 < i < n. \]
A $2$-rigid subcategory will be called \emph{rigid}.
We say that $\T$ is \emph{$m$-strong} if 
\[ \Hom_\C(\T[i],\T)=0 \text{ for all } 0 < i < m. \]
Note that this is equivalent to $\Hom_\C(\T,\T[-i])=0$ for all $0 < i < m$.

A morphism $f:T\to X$ with $T\in\T$ and $X\in\C$ is called a \emph{right $\T$-approximation of $X$} if 
\[ \xymatrix{ \left.\Hom_\C(-,T)\right|_\T \ar[rr]^-{\Hom_\C(-,f)} && \left.\Hom_\C(-,X)\right|_\T \ar[r] & 0 } \]
is exact as functors on $\T$.
We say that $\T$ is a \emph{contravariantly finite subcategory of $\C$} if every $X\in\C$ admits a right $\T$-approximation.
Left $\T$-approximations and covariantly finite subcategories are defined dually.
A subcategory $\T$ of $\C$ is called \emph{functorially finite} if it is both contravariantly finite and covariantly finite.

Let $\X$ and $\Y$ be subcategories of $\C$. Denote by $\X\ast\Y$ the full subcategory of $\C$ formed by all objects $Z$ appearing
in a triangle
\[ \xymatrix{ X \ar[r] & Z \ar[r] & Y \ar[r] & X[1] } \]
with $X\in\X$ and $Y\in\Y$.
Although $\X\ast\Y$ is not necessarily closed under direct summands, this is the case when $\Hom_\C(\X,\Y)=0$ (see \cite[Proposition 2.1]{IY}).
\begin{lemma}\label{starnohom}
Let $\T,\X$ and $\Y$ be subcategories of $\C$. If $\Hom_\C(\T,\X)=0$ and $\Hom_\C(\T,\Y)=0$, then $\Hom_\C(\T,\X\ast\Y)=0$.
\end{lemma}
\begin{proof}
Let $Z\in\X\ast\Y$, then there is a triangle
\[ \xymatrix{ X \ar[r]^{f} & Z \ar[r]^{g} & Y \ar[r] & X[1] } \]
with $X\in\X$ and $Y\in\Y$. 
For every $T\in\T$ we have an exact sequence
\[ \xymatrix{ \Hom_\C(T,X) \ar[r] & \Hom_\C(T,Z) \ar[r] & \Hom_\C(T,Y) } \]
\end{proof}

\subsection{Cohomological Functor} 
We denote by $\mod\T$ the category of finitely presented contravariant functors on $\T$.

Let $\T$ be a contravariantly finite subcategory of $\C$. The category $\mod\T$ is abelian and has enough projectives.
The functor $\H:\C\to\mod\T$ defined by $\H(X) = \left.\Hom_\C(-,X)\right|_\T$ induces an equivalence between $\T$ and the projective objects of $\mod\T$. Note that a morphism $f$ is a right $\T$-approximation if and only if $\H(f)$ is an epimorphism.

Let $\mathcal{A}$ be an additive category and $\mathcal{B}$ a full subcategory of $\mathcal{A}$ closed under isomorphisms, direct sums and direct summands. The quotient category $\mathcal{A}/\mathcal{B}$ has the same objects as $\mathcal{A}$ and $\Hom_{\mathcal{A}/\mathcal{B}}(X,Y)$ is the quotient of $\Hom_{\mathcal{A}}(X,Y)$ by the subgroup of morphisms factoring thorough some object in $\mathcal{B}$.
\begin{proposition}[{\cite[Proposition 6.2(3)]{IY}}]\label{equivmodt}
Let $\T$ be a rigid subcategory of $\C$.
The functor $\H$ induces an equivalence $\left(\T\ast\T[1]\right)/\T[1] \tilde{\longrightarrow}\mod\T$.
\end{proposition}

The following is a variation of a well-known result on quotient categories.
A proof is included for the convenience of the reader.
\begin{corollary}
Let $\T$ be a rigid subcategory of $\C$ and let $X,Y\in\T\ast\T[1]$.
If $\H X\cong \H Y$, then there exists $T\in\T$ such that $X$ is isomorphic to a direct summand of $Y\oplus T[1]$.
\end{corollary}
\begin{proof}
Let $\H(f):\H X\to \H Y$ be an isomorphism. There exists $g:Y\to X$ such that $\H(g)\H(f) = \mathbb{I}_{\H X}$.
Then $\H(\mathbb{I}_X - gf) = 0$ hence $\mathbb{I}_X - gf$ factors through an object in $\T[1]$.
Thus, there are morphisms $\alpha:X\to T[1]$ and $\beta:T[1]\to X$ with $T\in\T$ such that $\mathbb{I}_X - gf = \beta\alpha$.
We have
\[ \left[\begin{array}{cc} g & \beta \end{array}\right] \left[\begin{array}{c} f \\ \alpha \end{array}\right]
  = gf + \beta\alpha = \mathbb{I}_X \]
Therefore $\left[\begin{smallmatrix} f \\ \alpha \end{smallmatrix}\right]:X\to Y\oplus T[1]$ is a section.
\end{proof}

\subsection{Cluster-Tilting Subcategories} 
A subcategory $\T$ of $\C$ is called an \emph{$n$-cluster-tilting subcategory} if it satisfies the following conditions:
\begin{defenum}
  \item $\T$ is functorially finite.
  \item $X\in\T$ if and only if $\Hom_\C(\T,X[i])=0$ for all $0 < i < n$.
  \item $X\in\T$ if and only if $\Hom_\C(X,\T[i])=0$ for all $0 < i < n$.
\end{defenum}
A $2$-cluster-tilting subcategory will simply be called \emph{cluster-tilting}.

\begin{lemma}[{\cite[Lemma 3.2]{KZ}}]
Let $\T$ be contravariantly finite. If
\[ \Hom_\C(\T,X[1])=0 \text{ if and only if } X\in\T \]
then $\T$ is cluster-tilting.
\end{lemma}

\begin{theorem}[{\cite[Theorem 3.1]{IY}}]
Let $\T$ be an $n$-cluster-tilting subcategory of $\C$. Then
$\C = \T\ast\T[1]\ast\cdots\ast\T[n-1]$. 
\end{theorem}

Let $\A$ be an abelian category having enough projectives and enough injectives.
The category $\A$ is said to be \emph{Gorenstein of dimension $n$} if
the maximum of the injective dimensions of projectives and the projective dimensions
of injectives is equal to $n$.
\begin{theorem}[{\cite[Theorem 4.3]{KZ}}]\label{KZ:4.3}
Let $\T$ be a cluster-tilting subcategory of $\C$.
The category $\mod\T$ is Gorenstein of dimension at most one.
\end{theorem}

%
\section{Factorisation Ideals}\label{fi}
%

Let $\X$ and $\Y$ be subcategories of a category $\C$.
For every $X\in\X$, $Y\in\Y$ and $M\in\C$, we define a subset $\I_{M}(X,Y)$ of $\Hom_\C(X,Y)$ by
\[\I_{M}(X,Y) = \set{f:X\to Y}{f \text{ factors through } M}\]
Write $\I_{M}(\X,\Y)=0$ precisely when $\I_{M}(X,Y)=0$ for every $X\in\X$ and $Y\in\Y$.

When $\X=\Y$ we simply write $\I_{M}(\X)$ instead of $\I_{M}(\X,\X)$.
In this case, $\I_{M}(\X)$ is an ideal of the category $\X$.

\begin{lemma}\label{fi:ds}
Let $M$ and $N$ be objects in $\C$. If $M$ is isomorphic to a direct summand of $N$, then
$\I_{M}(\X,\Y) \subseteq \I_{N}(\X,\Y)$.
\end{lemma}
\begin{proof}
Suppose that $M$ is isomorphic to a direct summand of $N$.
Let $\iota: M \to N$ denote the corresponding section with retraction $\rho:N \to M$.
For any $X\in\X$ and $Y\in\Y$, let
\[ \xymatrix{ X \ar[r]^{\alpha} & M \ar[r]^-{\beta} & Y} \] be an element of $\I_{M}(X,Y)$.
We consider the composition $\beta\rho \circ \iota\alpha$
\[ \xymatrix{X \ar[r]^{\alpha}\ar[rd]_{\iota\alpha} & M \ar[r]^-{\beta}\ar@<-2pt>[d]_{\iota}     & Y \\
                                                    & N \ar[ru]_{\beta\rho} \ar@<-2pt>[u]_{\rho} &   \\ } \]
We have $\beta\alpha = \beta\rho\iota\alpha$ since $\rho\iota = \mathbb{I}_{M}$. Thus $\beta\alpha\in\I_{N}(X,Y)$.
\end{proof}

Let $\T$ be a contravariantly finite subcategory of $\C$. For any $X\in\C$, let $g_0:T_0\to X$ be a right $\T$-approximation of $X$.
Complete $g_0$ to a triangle
\[ \xymatrix{ \Omega^{1}X \ar[r] & T_0 \ar[r]^-{g_0} & X \ar[r] & \left(\Omega^{1}X\right)[1] } \]
Continuing inductively, we obtain a triangle for every $i\geq 0$
\[ \xymatrix{ \Omega^{i+1}X \ar[r] & T_i \ar[r]^-{g_i} & \Omega^{i}X \ar[r] & \left(\Omega^{i+1}X\right)[1] } \]
where $g_i$ is a right $\T$-approximation of $\Omega^{i}X$ and we set $\Omega^{0}X=X$.
We remark that the objects $\Omega^{i}X$ are not uniquely determined.
\begin{proposition}\label{fi:os}
Let $\D$ be a subcategory of $\C$ and let $\T$ be contravariantly finite in $\C$.
\begin{thmenum}
  \item Let $\T$ be $n$-rigid. For every $0 \leq k < n-1$ and every $i\geq 0$,
  \[ \text{if } \I_{\Omega^{i+1} X}(\D,\T[k]) = 0, \text{ then } \I_{\Omega^i X}(\D[1],\T[k+1]) = 0. \]
  \item Let $\T$ be $m$-strong. For every $0<k\leq m-1$ and every $i\geq 0$,
  \[ \text{if } \I_{\Omega^i X}(\T[k+1],\D[1]) = 0, \text{ then } \I_{\Omega^{i+1} X}(\T[k],\D) = 0. \]
\end{thmenum}
\end{proposition}
\begin{proof}
We will only prove (a), the proof of (b) is similar.
Let $\alpha:D[1]\to \Omega^i X$ and $\beta:\Omega^i X\to T[k+1]$ be any morphisms with $D\in\D$ and $T\in\T$.
There is a triangle
\[ \xymatrix{ \Omega^{i+1} X \ar[r] & T_i \ar[r]^-{g_i} & \Omega^i X \ar[r]^-{h_i} & \left(\Omega^{i+1} X\right)[1] } \]
with $T_i\in\T$. We have $\beta g_i \in \Hom_\C(T_i,T[k+1])$, thus $\beta g_i=0$ since $\T$ is $n$-rigid and $0 < k+1 < n$.
Hence there exists $u:\left(\Omega^{i+1} X\right)[1]\to T[k+1]$ making the following diagram commutative:
\[ \xymatrix{
D \ar[d]^{h_i\alpha[-1]}\ar[r] & 0 \ar[d]\ar[r] & D[1] \ar[d]^{\alpha} \ar@{=}[r] & D[1] \ar[d]^{h_i\alpha} \\
\Omega^{i+1} X \ar[r]\ar[d]^{u[-1]} & T_i \ar[d]\ar[r]^{g_i} & \Omega^i X \ar[d]^{\beta}\ar[r]^-{h_i} & \left(\Omega^{i+1} X\right)[1] \ar[d]^u \\
T[k] \ar[r] & 0 \ar[r] & T[k+1] \ar@{=}[r] & T[k+1] \\
} \]
Since $(uh_i\alpha)[-1]\in\I_{\Omega^{i+1} X}(\D,\T[k]) = 0$, we conclude that $\beta\alpha=0$. 
\end{proof}

\begin{corollary}\label{fiomegacor}
Let $\T$ be a contravariantly finite subcategory of $\C$.
\begin{thmenum}
  \item Let $\T$ be $n$-rigid.
  If $\I_{\Omega^{n-2} X}(\T[1]) = 0$, then $\I_{X}(\T[n-1]) = 0$.
  \item Let $\T$ be $m$-strong.
  If $\I_{X}(\T[m]) = 0$, then $\I_{\Omega^{m-1} X}(\T[1]) = 0$.
\end{thmenum}
\end{corollary}

\begin{proposition}
Let $\T$ be an $n$-cluster-tilting subcategory of $\C$. Then
$X\in\T$ if and only if $\I_{X}(\T[-n+1]\ast\cdots\ast\T[-1],\T[1]\ast\cdots\ast\T[n-1])= 0$.
\end{proposition}
\begin{proof}
For simplicity let $\Y = \T[-n+1]\ast\cdots\ast\T[-1]$ and $\Z = \T[1]\ast\cdots\ast\T[n-1]$.
Suppose that $X\in\T$. 
Since $\T$ is $n$-rigid, $\Hom_\C(X,\T[i])=0$ for each $0<i<n$. Hence $\Hom_\C(X,\Z)=0$ by \cref{starnohom}.
Thus, any factorisation
\[ \xymatrix{Y \ar[r] & X \ar[r] & Z} \] with $Y\in\Y$ and $Z\in\Z$ must be zero.

Conversely, suppose that $\I_{M}(\Y,\Z)= 0$.
Since $\T$ is $n$-cluster-tilting, we have $\C=\T\ast\T[1]\ast\cdots\ast\T[n-1]=\T\ast\Z$, thus there is a triangle
\[ \xymatrix{ Z\ar[r] & T \ar[r]^-{f} & X \ar[r]^-{g} & Z[1] } \]
with $T\in\T$ and $Z\in\Z$.
Let $h\in\Hom_\C(Y, X)$ with $Y$ any object in $\Y$.
Then the composition $gh$ is zero since $gh\in\I_{M}(\Y,\Z)$. 
We deduce the existence of a commutative diagram
\[ \xymatrix{
0 \ar[r]\ar[d] & Y \ar@{=}[r] \ar@{-->}[d]^{u} & Y \ar[r]\ar[d]^{h} & 0    \ar[d] \\
Z \ar[r]       & T \ar[r]^-{f}                 & X \ar[r]^-{g}      & Z[1]        \\ }
\]
We have $\Hom_\C(\T[-i],\T)\cong \Hom_\C(\T,\T[i])=0$ for every $0<i<n$, which implies $\Hom_\C(\Y,\T)=0$.
Since $u\in\Hom_\C(Y,T)=0$, we get $h=fu=0$.
Hence $\Hom_\C(\T[-i],X)=0$ for each $0<i<n$, from which $X\in\T$.
\end{proof}

%
\section{Rigid Subcategories}
%
Throughout this section, $\T$ is a rigid contravariantly finite subcategory of $\C$. 

\begin{proposition}\label{pd1onlyif}
Let $X\in\T\ast\T[1]$. If $\I_X(\T[1])$ is zero, then $\pd \H X \leq 1$.
\end{proposition}
\begin{proof}
Since $X\in\T\ast\T[1]$, there is a triangle
\[ \xymatrix{ X[-1]\ar[r]^-{f} & T_1 \ar[r]^-{g} & T_0 \ar[r] & X } \]
with $T_0,T_1\in\T$. Applying $\H$ to this triangle yields an exact sequence in $\mod\T$
\[ \xymatrix{ \H X[-1] \ar[r]^-{\H(f)} & \H T_1 \ar[r]^-{\H(g)} & \H T_0 \ar[r] & \H X \ar[r] & 0 } \]
Let $h:T\to X[-1]$ be a right $\T$-approximation of $X[-1]$. The composition
\[ \xymatrix{ T[1] \ar[r]^-{h[1]} & X \ar[r]^-{f[1]} & T_1[1] } \]
is an element of $\I_X(\T[1])$, thus $fh=0$.
Then $\H(fh)=0$ and since $\H(h)$ is an epimorphism, we get $\H(f)=0$.
Hence the projective dimension of $\H X$ is at most one. 
\end{proof}

\begin{theorem}\label{pd1iff}
Let $X\in\T\ast\T[1]$ having no direct summands in $\T[1]$. Then $\pd \H X \leq 1$ if and only if $\I_X(\T[1])$ is zero.  
\end{theorem}
\begin{proof}
Sufficiency follows directly from \cref{pd1onlyif}.

Suppose that $\pd \H X = 0$. Then $X$ belongs to $\T$, and every composition of morphisms
\[ \xymatrix{T[1] \ar[r] & X \ar[r]^-{\beta} & T'[1]} \]
with $T,T'\in\T$ must be zero since $\beta \in \Hom_C(\T,\T[1])=0$.
Therefore $\I_X(\T[1])=0$.

Now, assume $\H X$ has projective dimension $1$. There is a projective resolution
\[ \xymatrix{ 0 \ar[r] & \H T_1 \ar[r]^{\H(f)} & \H T_0 \ar[r] & \H X \ar[r] & 0} \]
in $\mod\T$ with $T_0,T_1\in\T$. Complete the morphism $f$ to a triangle
\[ \xymatrix{ T_1 \ar[r]^{f} & T_0 \ar[r] & Y \ar[r] & T_1[1] } \]
Since $T_1[1]\in\T[1]$, we get an exact sequence
\[ \xymatrix{ \H T_1 \ar[r]^{\H(f)} & \H T_0 \ar[r] & \H Y \ar[r] & 0 } \]
We conclude that $\H X$ and $\H Y$ are isomorphic in $\mod\T$.
The objects $X$ and $Y$ both belong to $\T\ast\T[1]$. Hence, in the category $\C$, $X$ is isomorphic to a direct summand of $Y\oplus T_2[1]$ for some $T_2\in\T$. Since $X$ has no direct summands in $\T[1]$, $X$ is isomorphic to a summand of $Y$.
Our aim is to show that $\I_X(\T[1])=0$ but, by \cref{fi:ds}, it is sufficient to establish $\I_Y(\T[1])=0$.
Given a factorisation
\[ \xymatrix{ T[1] \ar[r]^{\alpha} & Y \ar[r]^-{\beta} & T'[1]} \]
with $T,T'\in\T$, we will show that $\beta\alpha =0$.
We construct a commutative diagram where the rows are triangles
\[ \xymatrix{ T   \ar[r]\ar[d]^{g}     & 0   \ar[r]\ar[d] & T[1]  \ar@{=}[r]\ar[d]^{\alpha} & T[1]   \ar[d]^{g[1]} \\
              T_1 \ar[r]^{f}\ar[d]^{h} & T_0 \ar[r]\ar[d] & Y     \ar[r]\ar[d]^{\beta}      & T_1[1] \ar[d]^{h[1]} \\
              T'  \ar[r]               & 0   \ar[r]       & T'[1] \ar@{=}[r]                & T'[1]                \\ } \]
and the existence of $g$ follows from the commutativity of the upper-middle square.
Moreover, the lower-middle square commutes since $\Hom_\C(\T,\T[1])=0$, assuring the existence of the morphism $h$.
From the commutativity of this diagram we deduce that $fg=0$, and hence $\H(g)=0$ since $\H(f)$ is a monomorphism. Thus $g=0$ since $\H$ induces an equivalence on $\T$.
Therefore, $\beta\alpha = h[1] \circ g[1] = 0$.
\end{proof}

It is important to note that, by \cref{equivmodt}, every object in $\mod\T$ is isomorphic to some $\H X$ with $X\in\T\ast\T[1]$ having no direct summands in $\T[1]$.
For any $X\in\C$ define $\overline{X}$ to be the direct sum of all the indecomposable direct summands of $X$ not contained in $\T[1]$.
In particular, $\overline{X} = 0$ whenever $X\in\T[1]$.
\begin{corollary}\label{pdinfty}
Let $\T$ be a cluster-tilting subcategory of $\C$.
\begin{thmenum}
  \item\label{pdinfty1} For every $X\in\C$ having no direct summands in $\T[1]$, 
        \[\pd \H X = \infty\text{ if and only if }\I_X(\T[1])\neq 0.\]
  \item For every $X\in\C$, $\pd \H X = \infty$ if and only if $\I_{\overline{X}}(\T[1])\neq 0$.
\end{thmenum}
\end{corollary}
\begin{proof}\mbox{}
\begin{thmenum}
  \item Since $\T$ is cluster-tilting, we have $\C=\T\ast\T[1]$. Thus, for every $X\in\C$ having no direct summands in $\T[1]$,
        $\pd \H X > 1$ if and only if $\I_X(\T[1])\neq 0$.
        By \cref{KZ:4.3}, $\mod\T$ is Gorenstein of dimension at most $1$, so $\pd \H X > 1$ if and only if $\pd \H X = \infty$.
  \item Let $X\in\C$. By \ref{pdinfty1}, $\pd \H \overline{X} = \infty$ if and only if $\I_{\overline{X}}(\T[1])\neq 0$.
        Clearly $\H X \cong \H \overline{X}$, thus $\pd \H \overline{X} = \infty$ if and only if $\pd \H X = \infty$.
\end{thmenum}
\end{proof}

%
\section{Some Generalisations}
%
%
When $\T$ be an $n$-cluster-tilting subcategory of $\C$ with $n>2$, the category $\mod\T$ does not, in general, have the Gorenstein property.
However, this problem can be fixed with an additional assumption.
The following is a special case of Theorem 7.5 from \cite{B}.
\begin{theorem}
Let $\T$ be an $n$-cluster-tilting subcategory of $\C$.
If $\T$ is $(n-1)$-strong, then $\mod\T$ is Gorenstein of dimension at most one.
\end{theorem}
Whence we immediately get a generalisation of \cref{pdinfty}:
\begin{theorem}
Let $\T$ be an $(n-1)$-strong $n$-cluster-tilting subcategory of $\C$ and let 
$X\in\T\ast\T[1]$ having no direct summands in $\T[1]$. Then 
\[\pd \H X = \infty\text{ if and only if }\I_X(\T[1])\neq 0.\]
\end{theorem}

For $X\in\C$, let $\Omega^{i}X$ be the objects defined in section \ref{fi}.
\begin{lemma}
Let $\T$ be an $n$-rigid subcategory of $\C$. 
If $X\in\T\ast\T[1]\ast \cdots \ast \T[n-1]$, then we may choose $\Omega^{i}X$ to be in $\T\ast\T[1]\ast \cdots \ast \T[n-i-1]$ for each $0\leq i < n$.
\end{lemma}
\begin{proof}
Since $\Omega^0 X = X \in \T\ast\T[1]\ast \cdots \ast \T[n-1]$ there is a triangle
\[
  \xymatrix{ Z_0 \ar[r] & T_0 \ar[r]^-{g_0} & X \ar[r] & Z_0[1] }
\]
with $T_0\in\T$ and $Z_0[1]\in\T[1]\ast \cdots \ast \T[n-1]$.
Our assumption that $\T$ is $n$-rigid along with \cref{starnohom} implies that $\Hom_\C(\T,\T[1]\ast \cdots \ast \T[n-1])=0$.
Thus $\Hom_\C(\T,Z_0[1])=0$ from which we deduce that $g_0$ is a right $\T$-approximation of $X$.
Hence we may choose $\Omega^{1}X=Z_0 \in \T\ast \cdots \ast \T[n-2]$.
The claim follows by applying this same reasoning inductively.
\end{proof}
If, in addition, $\T$ is $(n-1)$-strong, then it is easy to see that $X$ having no direct summands in $\T[1]\ast\cdots\ast\T[n-1]$ implies that each $\Omega^{i}X$ has no direct summands in $\T[1]\ast \cdots \ast \T[n-i-1]$.
In fact, suppose that $\Omega^{1}X = Z_1 \oplus Z_2$ with $Z_2\in\T[1]\ast\cdots\ast\T[n-2]$.
There is a triangle
\[ \xymatrix{ T_0 \ar[r] & X \ar[r] & (Z_1 \oplus Z_2)[1] \ar[r] & T_0[1] } \]
with $T_0\in\T$. Since $\T$ is $(n-1)$-strong, we have $\Hom_\C(Z_2[1],T_0[1])=0$. Thus $Z_2[1]$ is isomorphic to a direct summand of $X$, a contradiction.
The claim then follows by induction.
With this choice of $\Omega^{i}X$, we have the following:
\begin{proposition}
Let $\T$ be $n$-rigid and $(n-1)$-strong and let $X\in\T\ast\T[1]\ast \cdots \ast \T[n-1]$ having no direct summands in $\T[1]\ast\cdots\ast\T[n-1]$.
Then
\[ \I_X(\T[n-1])=0 \text{ if and only if } \pd \H \Omega^{n-2}X\leq 1.\] 
\end{proposition}
\begin{proof}
Suppose that $\T$ is both $n$-rigid and $(n-1)$-strong. Then by \cref{fiomegacor}, 
\[ \I_{\Omega^{n-2} X}(\T[1]) = 0 \text{ if and only if } \I_{X}(\T[n-1]) = 0\]
Since $\Omega^{n-2} X\in\T\ast\T[1]$ having no direct summands in $\T[1]$, we have, by \cref{pd1iff}
\[ \I_{\Omega^{n-2} X}(\T[1]) = 0 \text{ if and only if } \pd \H \Omega^{n-2}X\leq 1 \]
\end{proof}

\begin{lemma}\label{g:l1}
Let $\T$ be $n$-rigid and $(n-1)$-strong. If $X\in\T\ast\T[1]\ast\cdots\ast\T[n-1]$ and
$\H X[-i]=0$ for all $0 < i < n$, then $\pd \H X\leq n-1$.
\end{lemma}
\begin{proof}
We begin by remarking that if $\T$ is $n$-rigid, then $\T$ is also $k$-rigid for all $2\leq k \leq n$ (an similarly for the $m$-strong property).

The proof is by induction on $n$. For $n=2$, we have $X\in\T\ast\T[1]$ thus there is a triangle
\[ \xymatrix{ T_1\ar[r] & T_0 \ar[r] & X \ar[r] & T_1[1] } \]
with $T_0,T_1\in\T$. Applying $\H$ to this triangle yields an exact sequence
\[ \xymatrix{ \H X[-1]\ar[r] & \H T_1\ar[r] & \H T_0 \ar[r] & \H X \ar[r] & 0 } \]
Since $\H X[-1] = 0$ we conclude that $\pd \H X\leq 1$.

Assume the statement holds for $n-1$.
Let $X\in\T\ast\T[1]\ast\cdots\ast\T[n-1]$.
There is a triangle
\[ \xymatrix{ Y \ar[r] & T \ar[r] & X \ar[r] & Y[1] } \]
with $T\in\T$ and $Y\in\T\ast\T[1]\ast \cdots \ast \T[n-2]$.
For every $i$ we have an exact sequence
\[ \xymatrix{ \H X[-i-1]\ar[r] & \H Y[-i]\ar[r] & \H T[-i] } \]
Since $\T$ is $(n-1)$-strong, $\H T[-i]=0$ for $0 < i < n-1$. Also $\H X[-i-1]=0$ for each $0 < i < n-1$
thus $\H Y[-i]=0$ for all $0 < i < n-1$. Hence $\pd\H Y\leq n-2$ by the inductive hypothesis.

Moreover, using the fact that $\H X[-1] = 0$, we have a short exact sequence
\[ \xymatrix{ 0\ar[r] & \H Y\ar[r] & \H T\ar[r] & \H X\ar[r] & 0} \]
with $\H T$ projective. Therefore $\pd \H X\leq n-1$.
\end{proof}

\begin{proposition}\label{nrpd1onlyif}
Let $\T$ be $n$-rigid and $(n-1)$-strong and let $X\in\T\ast\T[1]\ast\cdots\ast\T[n-1]$.
If $\I_X(\T[1]\ast\cdots\ast\T[n-1])=0$, then $\pd \H X\leq n-1$.
\end{proposition}
\begin{proof}
Since $X\in\T\ast\T[1]\ast\cdots\ast\T[n-1]$, there is a triangle
\[ \xymatrix{ Y \ar[r] & T_0 \ar[r] & X \ar[r]^-{h} & Y[1] } \]
with $T_0\in\T$ and $Y[1]\in\T[1]\ast\cdots \ast \T[n-1]$.
Applying $\H$ to this triangle yields an exact sequence in $\mod\T$
\[ \xymatrix{ \H X[-i] \ar[rr]^-{\H(h[-i])} && \H Y[-i+1] \ar[r] & \H T_0[-i+1]} \]
For $1\leq i \leq n-1$, let $f:T\to X[-i]$ be a right $\T$-approximation of $X[-i]$. The composition
\[ \xymatrix{ T[i] \ar[r]^-{f[i]} & X \ar[r]^-{h} & Y[1] } \]
is an element of $\I_X(\T[1]\ast\cdots\ast\T[n-1])$, thus $h[-i]\circ f=0$.
Then $\H(h[-i]\circ f)=0$ and since $\H(f)$ is an epimorphism, we get $\H(h[-i])=0$.
Because $\T$ is $(n-1)$-strong, we have $\H T_0[-i]=0$ for $0 < i < n-1$.
Therefore $\H Y[-i]=0$ for $0 < i < n-1$ and since $Y\in\T\ast\T[1]\ast\cdots\ast\T[n-2]$ 
we may apply \cref{g:l1} to conclude that $\pd \H Y \leq n-2$.
Additionally, since $\H(h[-1])=0$, we have a short exact sequence
\[ \xymatrix{ 0\ar[r] & \H Y\ar[r] & \H T_0\ar[r] & \H X\ar[r] & 0} \]
thus the projective dimension of $\H X$ is at most $n-1$.
\end{proof}
By comparing \cref{pd1iff} and \cref{nrpd1onlyif}, it is natural to conjecture that
if $\T$ is $n$-rigid and $(n-1)$-strong and $X\in\T\ast\T[1]\ast\cdots\ast\T[n-1]$ containing no direct summands in $\T[1]\ast\cdots\ast\T[n-1]$, then
\[ \pd \H X\leq n-1 \text{ if and only if } \I_X(\T[1]\ast\cdots\ast\T[n-1])=0\]
We do not, however, have a proof of this claim.


\end{document}